\documentclass{amsart}

\usepackage{amscd} \usepackage{amssymb} \usepackage{latexsym}
\usepackage{graphicx}
\usepackage{amsmath} \usepackage{amsfonts} \usepackage{amssymb}
\usepackage{latexsym}
\usepackage{enumerate}
\usepackage{comment} \newcommand{\Z}{\mathbb{Z}}

\usepackage[utf8]{inputenc}

\usepackage[final]{hyperref} \hypersetup{colorlinks=true,
  linkcolor=blue, anchorcolor=blue, citecolor=red, filecolor=blue,
  menucolor=blue, urlcolor=blue}

\def\norm#1{\|#1\|}

\newcommand{\Bigabs}[1]{\Bigl\vert #1 \Bigr\vert}

\newcommand{\C}{\mathbb{C}} 
\newcommand{\R}{\mathbb{R}}

\newcommand{\angles}[1]{\langle #1 \rangle}

\DeclareMathOperator{\supp}{supp}

\newtheorem{theorem}{Theorem}
\newtheorem{corollary}[theorem]{Corollary}
\newtheorem{proposition}[theorem]{Proposition}
\newtheorem{lemma}[theorem]{Lemma}

\theoremstyle{remark} \newtheorem{remark}[theorem]{Remark}

\theoremstyle{definition} \newtheorem{definition}[theorem]{Definition}

\numberwithin{equation}{section} \numberwithin{theorem}{section}

\title[Scattering for semi-relativistic equations]{Small data
  scattering for semi-relativistic equations with Hartree type
  nonlinearity}

\author[S.~Herr]{Sebastian Herr} \address{Universit\"{a}t Bielefeld,
  Fakult\"{a}t f\"{u}r Mathematik, Postfach 10 01 31 , D-33501
  Bielefeld, Germany} \email{herr@math.uni-bielefeld.de}

\author[A.~Tesfahun]{Achenef Tesfahun} \address{Universit\"{a}t
  Bielefeld, Fakult\"{a}t f\"{u}r Mathematik, Postfach 10 01 31 ,
  D-33501 Bielefeld, Germany} \email{achenef@math.uni-bielefeld.de.}

\thanks{The authors acknowledge support from the German Research
  Foundation, Collaborative Research Center 701.  }
\begin{document}

\begin{abstract}
  We prove that the initial value problem for the equation $$ -
  i\partial_t u + \sqrt{m^2-\Delta} \, u= \left(\frac{e^{-\mu_0
        |x|}}{|x|} \ast |u|^2\right)u \ \text{in } \ \R^{1+3}, \quad
  m\ge 0, \ \mu_0 >0$$ is globally well-posed and the solution
  scatters to free waves asymptotically as $t \rightarrow \pm \infty$
  if we start with initial data which is small in $H^s(\R^3)$ for
  $s>\frac12$, and if $m>0$. Moreover, if the initial data is radially
  symmetric we can improve the above result to $m\ge 0$ and $s>0$,
  which is almost optimal, in the sense that $L^2(\R^{3})$ is the
  critical space for the equation.  The main ingredients in the proof
  are certain  endpoint Strichartz estimates, $L^2(\R^{1+3})$ bilinear
  estimates for free waves and an application of the $U^p$ and $V^p$
  function spaces.
 
\end{abstract}

\maketitle

\section{Introduction}
We consider the initial value problem (IVP) for the semi-relativistic
equation with a cubic Hartree-type nonlinearity:
\begin{equation}\label{BS}
  \begin{split}
    - i\partial_t u + \sqrt{m^2-\Delta} \, u&= (V \ast |u|^2)u \quad \text{in } \R^{1+3},\\
    u(0,\cdot)&=f \in H^s(\R^3),
  \end{split}
\end{equation}
where $\sqrt{m^2-\Delta}$ is defined via its symbol $\sqrt{m^2 +
  |\xi|^2}$ in Fourier space, the constant $m \geq 0$ is a physical
mass parameter, the symbol $\ast$ denotes convolution in $\R^3$ and
$V$ is a potential, typically,
\begin{align}
  \label{P}
  \quad V(x)&= \frac{e^{-\mu_0|x|}}{|x|} , \quad \mu_0\ge 0,
\end{align}
which is called a \emph{Coulomb potential} if $\mu_0=0$ and a
\emph{Yukawa potential} if $\mu_0>0$.

Equation \eqref{BS} is used to describe the dynamics and gravitational
collapse of relativistic boson stars and it is often referred to as
\textit{the boson star equation}; see \cite{ES07, FJL07, LHT07, MS12}
and the references therein.

It is well-known that equation \eqref{BS} exhibits the following
conserved quantities of energy and $L^2$-mass, which are given by
\begin{align*}
  E(u(t)) &= \frac{1}{2} \int_{\R^3} \bar{u} \sqrt{m^2-\Delta} u \, dx
  + \frac{1}{4} \int_{\R^3} \left( V \ast {|u|}^2 \right) {|u|}^2 \,
  dx \, ,
  \\
  M(u(t)) &= \int_{\R^3} |u|^2 \, dx.
\end{align*}
From these conservation laws, we see that the Sobolev space
$H^{\frac12}(\R^3)$ serve as the energy space for problem \eqref{BS}.
Furthermore, in the case of $m=0$, \eqref{BS} is invariant under the
scaling
 $$u(t,x) \mapsto u_\lambda(t,x) = \lambda^{\frac32} u(\lambda t, \lambda x)$$ for 
 fixed $\lambda > 0$. This scaling symmetry leaves the $L^2$-mass
 $M(u(t)) = M(u_\lambda(t))$ invariant, and so equation \eqref{BS} is
 \emph{$L^2$-critical}.

 There has been a considerable mathematical interest concerning the
 low regularity well-posedness (both local and global-in-time) and
 scattering theory of the initial value problem \eqref{BS} in the past
 few years.  A first well-posedness result was obtained by Lenzmann
 \cite{L07} for $s \geq \frac12$ using energy methods.  Moreover, he
 showed global well-posedness in $H^{\frac12}$ for initial data
 sufficiently small in $L^2$.  There has been further well-posedness
 and scattering results for equation \eqref{BS} with a more general
 potential, namely, $V(x)= |x|^{-\gamma}, \ \gamma\in (0, n)$; see
 eg. \cite{CO06,CO07, CO08, COSS09}. Recently, Pusateri \cite{FP14} proved a modified scattering result in the case of the Coulomb potential in dimension $n=3$, if $m>0$.

 Recently, Lenzmann and the first author \cite{HL14} proved local
 well-posedness for initial data with $s>\frac14$ (and $s>0$ if the
 data is radially symmetric). Moreover, these results are optimal up
 to end points, i.e., up to $s=\frac14$ (and $s=0$), in the framework
 of perturbation methods. Strichartz estimates and space-time bilinear
 estimates were the main ingredients.  This is in contrast with
 previous results where only energy methods and linear Strichartz
 estimates were used.

 The aim of this paper is to prove global existence and scattering of
 solutions to the IVP \eqref{BS}. Our main result is the following.
 
\begin{theorem}[Main Theorem]\label{MainThm} Let $\mu_0>0$ in \eqref{P}, i.e., $V$ is a Yukawa potential.
  Assume one of the following holds:
  \begin{enumerate}[(a)]
  \item \label{MainThma} $m\ge 0$, $s>0$ and $f$ is radially
    symmetric,
  \item \label{MainThmb} $m> 0$ and $s>\frac12$.
  \end{enumerate}
  Then, there exists $\delta > 0$ such that for all $f \in H^s(\R^3)$
  satisfying
$$
\norm{f}_{ H^s} < \delta,
$$
the IVP \eqref{BS} has a global solution (spatially radial solution if
$f$ is radial) $$ u \in C(\R, H^s(\R^3)).$$ Moreover, the solution
depends continuously on $f$ and scatters asymptotically as $t
\rightarrow \pm \infty$.  Furthermore, it is unique in some smaller
subspace of $C(\R, H^s(\R^3))$.
\end{theorem}

If $m>0$, due to the modified scattering result in \cite{FP14} (see also \cite[Theorem 4.1]{CO06}), it is known that the scattering result in Theorem \ref{MainThm} does not carry over to the case $\mu_0=0$, i.e.\ if $V$ is the Coulomb potential.

\begin{remark}\label{OtherHartreeType}
  Let us consider the IVP for the nonlinear Dirac equation with
  Hartree type nonlinearity:
  \begin{equation}\label{Dirac}
    \begin{split}
      ( - i\partial_t + \boldsymbol \alpha \cdot D + m \beta) \psi&=
      \lambda (V \ast |\psi|^2)\psi \quad \text{in } \R^{1+3},
      \\
      \psi(0,\cdot)&=\psi_0 \in H^s(\R^3),
    \end{split}
  \end{equation}
  where $D=-i \nabla$, $\psi: \R^{1+3} \rightarrow \C^4$ is the Dirac
  spinor regarded as a column vector, $\lambda \in \C$, and $\beta$
  and $\boldsymbol \alpha=(\alpha^1, \alpha^2, \alpha^3)$ are the
  Dirac matrices. We refer the reader to \cite{BH13} for a
  representation of the Dirac matrices and a recent result for a
  related problem with a cubic nonlinearity with null-structure.
 
  Equation \eqref{Dirac}, with a Coulomb potential $V$, was derived by
  Chadam and Glassey \cite{CG76} by uncoupling the Maxwell-Dirac
  equations under the assumption of vanishing magnetic field. Then in
  two space dimensions they showed existence of a unique global
  solution for smooth initial data with compact support. They also
  conjectured \cite[see pp. 507]{CG76} equation \eqref{Dirac} with a
  Yukawa potential $V$ can be derived by uncoupling the
  Dirac-Klein-Gordon equations, see also \cite{CG74,BH14} for certain
  global and scattering results in this context.  Later, Dias and
  Figueira \cite{DF89-2, DF89-1} proved existence of weak solution for
  \eqref{Dirac} with a Yukawa potential in the massless case
  ($m=0$). We now comment on how to conclude a similar result as in
  Theorem \ref{MainThm}\ref{MainThmb} for the IVP \eqref{Dirac} with a
  Yukawa potential.
 
  Following \cite{BH13}, we define the projections
$$ P_{\pm}(\xi) = \frac{1}{2} \left( I \pm \frac{1}{\angles{\xi}_m}[\xi \cdot \boldsymbol \alpha + m\beta]
\right),$$ where $$\angles{\xi}_m=\sqrt{|\xi|^2+m^2}.$$ Then $\psi =
\psi_+ + \psi_-$, where $\psi_\pm = P_\pm(D) \psi$. Now, if we apply
$P_\pm(D)$ to \eqref{Dirac}, the IVP transforms to
\begin{equation}\label{Dirac2}
  \left\{
    \begin{aligned}
      & \bigl( -i\partial_t + \angles{D}_m \bigr) \psi_+ =
      P_+(D)\left[(V \ast |\psi|^2)\psi\right], \quad
      \psi_+(0,\cdot)=\psi_0^+ \in H^s(\R^3),
      \\
      & \bigl( -i\partial_t - \angles{D}_m \bigr) \psi_- =
      P_-(D)\left[(V \ast |\psi|^2)\psi\right], \quad
      \psi_-(0,\cdot)=\psi_0^-\in H^s(\R^3),
    \end{aligned}
  \right.
\end{equation}
where $\psi_0^\pm=P_\pm(D)\psi_0$. Notice that these equations are of
the form \eqref{BS}, and hence an easy modification\footnote{One has
  to be careful in the radial case since the Dirac operator does not
  preserve spherical symmetry in the classical sense (see e.g. \cite{
    MNN2005, EO12}).  However, we do not consider this case here. }
of the proof of Theorem \ref{MainThm}\ref{MainThmb} will give the
following result.
\end{remark}

\begin{corollary}\label{RemarkCorollary} Let $\mu_0>0$ in \eqref{P} (i.e., $V$ is a Yukawa potential).
  Assume $m> 0$ and $s>\frac12$.  Then, there exists $\delta > 0$ such
  that for all $(\psi_0^+, \psi_0^-) \in H^s(\R^3)\times H^s(\R^3)$
  satisfying
$$
\norm{\psi_0^\pm}_{ H^s} < \delta,
$$
the IVP \eqref{Dirac2} has a global solution $$ (\psi_+, \psi_-) \in
C(\R, H^s(\R^3)\times H^s(\R^3)).$$ Moreover, the solution depends
continuously on $(\psi_0^+, \psi_0^-) $ and scatters asymptotically as
$t \rightarrow \pm \infty$.  Furthermore, it is unique in some smaller
subspace.
\end{corollary}

Recently, in \cite{BH13,BH14} small data scattering results for low
regularity initial data have been proven for the nonlinear Dirac
equation and a massive Dirac-Klein-Gordon system.  These problems
exhibit null-structure and a non-resonant behavior. Note that there
is also null-structure in \eqref{Dirac2}, so we expect that the
regularity threshold in Corollary \ref{RemarkCorollary} can be
lowered, but we do not pursue this here. We will use some ideas
introduced in \cite{BH13,BH14}, in particular localized Strichartz
estimates in the non-radial case, but otherwise the analysis differs
significantly.

The rest of the paper is organized as follows. In the next Section, we
give some notation, define the $U^p$ and $V^p$-spaces and collect
their properties.  In Section \ref{LinBi-Est}, we prove some linear
and bilinear estimates for solutions of free Klein-Gordon equation. In
Section \ref{ProofMainThm}, we state a key proposition and then give
the proof of Theorem \ref{MainThm}.  In Section
\ref{ProofKeyProp}--\ref{ProofKeyLemma}, we give the proof of the key
proposition.

\section{Notation, $U^p$ and $V^p$-spaces and their
  properties}\label{Sec-UV}

\subsection{Notation}
We denote the spatial Fourier transform by $\widehat{\cdot}$ or
$\mathcal F_x$ and the space-time Fourier transform by
$\widetilde{\cdot}$.  Frequencies will be denoted by Greek letters
$\mu$ and $\lambda$, which we will assume to be dyadic, that is of the
form $2^k$ for $k \in \Z$.

Consider an even function $\beta_1 \in C_0^\infty((-2, 2))$ such that
$\beta_1(s)=1 $ if $|s|\le 1$ , and define for $\lambda>1$,
$$
\beta_{\lambda}(s)=\beta_1\left(\frac{s}{\lambda}\right)-\beta_1\left(\frac{2s}{\lambda}\right).
$$
Thus, $\supp \beta_{1}= \{ s\in \R: |s|<2 \}$ whereas $\supp
\beta_{\lambda}= \{ s\in \R: \frac\lambda 2<|s|<2\lambda \}$ for
$\lambda>1$.  We define
\begin{equation*}
  P_{\lambda} u:=u_\lambda=\mathcal F_x^{-1}(\beta_\lambda(|\cdot|)\mathcal F_x u)  
  \quad \text{and} \quad \widetilde P_\lambda=P_{\frac{\lambda}2}+P_\lambda+P_{2\lambda}.
\end{equation*}

For fixed $m\ge 0$, we denote $ S_m(t)=e^{-it\angles{D}_m} $ to be the
linear propagator of the Boson star equation \eqref{BS} defined by
$$
\mathcal F_x(S_m(t)f)(\xi)=e^{-it\angles{\xi}_m}\widehat{f}(\xi).
$$

\subsection{$U^p$ and $V^p$ spaces }
These function spaces were originally introduced in the unpublished
work of Tataru on the wave map problem and then in Koch-Tataru
\cite{KT07} in the context of NLS. The spaces have since been used to
obtain critical results in different problems related to dispersive
equations (see eg. \cite{HHK09, HHK09-1, TS09, HTT11}) and they serve
as a useful replacement of $X^{s, b}$-spaces in the limiting cases.
For the convenience of the reader we list the definitions and some
properties of these spaces.

Let $\mathcal{Z}$ be the collection of finite partitions $-\infty <
t_0 < \cdots < t_K \leq \infty$ of $\R$. If $t_K=\infty$, we use the
convention $u(t_K) :=0$ for all functions $u:\R\to L^2$.  We use
$\chi_I$ to denote the sharp characteristic function of a set $I
\subset \R$.

\begin{definition} \label{DefUp} Let $1\leq p < \infty$.  A $U^p$-atom
  is defined by a step function $a:\R\to L^2$ of the form
$$
a(t) = \sum_{k = 1}^K \chi_{[t_{k-1}, t_k) (t)} \phi_{k - 1},
 $$
 where $$\{t_k\}_{k = 0}^K \in \mathcal{Z}, \quad \{\phi_k\}_{k =
   0}^{K-1} \subset L^2 \ \text{with} \ \sum_{k = 0}^{K-1}
 \|\phi_k\|_{L^2}^p = 1.$$ The atomic space $U^p(\R; L^2)$ is defined
 to be the collection of functions $u:\R\to L^2$ of the form
 \begin{equation} \label{Up} u = \sum_{j = 1}^\infty \lambda_j a_j, \
   \text{ where $a_j$'s are $U^p$-atoms and} \ \{\lambda_j\}_{j \in
     \mathbb{N}}\in \ell^1,
 \end{equation}
 with the norm
$$\|u\|_{U^p} : = \inf_{\text{representation \eqref{Up} }}  \sum_{j=1}^{\infty} |\lambda_j|.
$$
\end{definition}

\begin{definition} \label{DefVp} Let $1\leq p<\infty$.

  \begin{enumerate}
  \item \label{V:def} define $V^p(\R, L^2)$ as the space of all
    functions $v:\R\to L^2$ for which the norm
    \begin{equation}\label{Vp-norm}
      \|v\|_{V^p}:=\sup_{\{t_k\}_{k=0}^K \in \mathcal{Z}}
      \left(\sum_{k=1}^{K}
        \|v(t_{k})-v(t_{k-1})\|_{L^2}^p\right)^{\frac{1}{p}}
    \end{equation}
    is finite.
  
  \item \label{V-:def} Likewise, let $V^p_-(\R, L^2)$ denote the
    normed space of all functions $v:\R\to L^2$ such that
    $\lim_{t\rightarrow -\infty} v(t)=0$ and $\|v\|_{V^p} < \infty$,
    endowed with the norm \eqref{Vp-norm}.

  \item \label{Vcr:def} We let $V^p_{rc}(\R, L^2)$ ($V^p_{-,rc} (\R,
    L^2)$) denote the closed subspace of all right-continuous $V^p
    (\R, L^2)$ functions ($V^p_-(\R, L^2)$ functions).

  \end{enumerate}

\end{definition}

We collect some useful properties of these spaces. For more details
about the spaces and proofs we refer to \cite{HHK09, HTT11}.

\begin{proposition}\label{Prop-Up}
  Let $1\leq p< q < \infty$. Then we have the following:
  \begin{enumerate}
  \item \label{U-Banach} $U^p(\R, L^2)$ is a Banach space.
  \item\label{U-Emb} The embeddings $U^p(\R, L^2)\subset U^q(\R,
    L^2)\subset L^\infty(\R;L^2)$ are continuous.
  \item\label{U-RightCont} Every $u\in U^p(\R, L^2)$ is
    right-continuous. Moreover, $\lim_{t\to - \infty}u(t)=0$.
  
  \end{enumerate}
\end{proposition}
 
\begin{proposition}\label{Prop-Vp}Let $1\leq p<q <\infty$. Then we
  have the following:
  \begin{enumerate}
  \item\label{V-Spaces} The spaces $V^p(\R, L^2)$, $V^p_{rc}(\R,
    L^2)$, $V^p_-(\R, L^2)$ and $V^p_{-,rc}(\R, L^2)$ are Banach
    spaces.
  \item \label{V-emb1} The embedding $U^p (\R, L^2) \subset V_{-,rc}^p
    (\R, L^2)$ is continuous.
  \item \label{V-emb2} The embeddings $V^p (\R, L^2)\subset V^q (\R,
    L^2)$ and $V^p_-(\R, L^2) \subset V^q_- (\R, L^2)$ are continuous.
  \item \label{V-embCont} The embedding $V^p_{-,rc}(\R, L^2) \subset
    U^q (\R, L^2)$ is continuous.
  \end{enumerate}
\end{proposition}

\begin{lemma}\cite{KTV14}\label{Lemma:Interp}
  Let $p>2$ and $v\in V^2(\R, L^2)$. There exists $\kappa=\kappa(p)>0$
  such that for all $M\ge 1$, there exist $w \in U^2(\R, L^2)$ and $z
  \in U^p(\R, L^2)$ with
  $$ v=   w+   z$$ and 
\begin{equation*}
  \frac{\kappa}M \norm { w}_{U^2}+ e^M
  \norm { z}_{U^p}\lesssim \norm { v}_{V^2}.
\end{equation*}
\end{lemma}

We now introduce $U^ p,V^p$-type spaces that are adapted to the linear
propagator $ S_m(t)=e^{it\angles{D}_m} $ of equation \eqref{BS}:

\begin{definition}\label{UV-Propag}
  We define $U^p_{m}(\R, L^2) $ (and $V^p_{m} (\R, L^2)$,
  respectively) to be the spaces of all functions $u: \R \mapsto
  L^2(\R^3)$ such that $t\to S_m(-t)u $ is in $U^p (\R, L^2) $
  (resp. $V^p(\R, L^2) $), with the respective norms:
  \begin{align*}
    \|u \|_{U^p_{m} } &= \| S_m(-t)u\|_{U^p},
    \\
    \|u \|_{V^p_{m} } &= \| S_m(-t) u\|_{V^p}.
  \end{align*}
  We use $V^p_{\text{rc}, m} (\R, L^2) $ to denote the subspace of
  right-continuous functions in $V^p_{m} (\R, L^2)$.

\end{definition}

 \begin{remark}
   Proposition \ref{Prop-Up}, Proposition \ref{Prop-Vp} and Lemma
   \ref{Lemma:Interp} naturally extends to the spaces $U^p_{m}(\R,
   L^2) $ and $V^p_{m} (\R, L^2)$.
 \end{remark}

\begin{lemma}\label{LemmaTransfer}
{\rm (Transfer principle)}
Let $$T:L^2\times \cdots \times L^2\to L^1_{loc}(\R^n;\C)$$ be a
multilinear operator and suppose that we have
\[ \big\| T(S_m(t) \phi_1, \dots, S_m(t) \phi_k)\big\|_{L^p_t
  L^r_x(\R\times \R^n)} \lesssim \prod_{j = 1}^k \|\phi_j\|_{L^2_x (
  \R^n)}\] for some $1\leq p, r \leq \infty$.  Then
\[ \big\| T(u_1, \dots, u_k)\big\|_{ L^p_t L^r_x(\R\times \R^n)}
\lesssim \prod_{j = 1}^k \|u_j\|_{U^p_{m} }.\]
\end{lemma}

\section{Linear and bilinear estimates}\label{LinBi-Est}
In this section, we prove linear and bilinear estimates for free
solutions of the Klein-Gordon equation both for radial and non-radial
data, which are key to prove our main Theorem.

\subsection{Estimates in the radial case}

\begin{lemma}\label{Lemma-KGStrRadial}
  Let $m\ge 0$.  Consider $u(t)=S_m(t) f$, where $f$ is radial.  Then
  \begin{align}
    \label{LinStr-radial}
    \norm{ u_{\lambda} }_{ L^2_{t} L^\infty_{ x} (\R^{1+3}) }&\lesssim
    \lambda \norm{ f_{ \lambda} }_{ L^2_{ x}(\R^{3} )}.
  \end{align}
  Moreover, for all radial function $u_\lambda \in U^2_{m} $, we have
  \begin{align}
    \label{LinStr-RadialTransfer}
    \norm{ u_{\lambda} }_{ L^2_{t} L^\infty_{ x} (\R^{1+3}) }&\lesssim
    \lambda \norm{ u_{\lambda} }_{ U^2_{m} }.
  \end{align}

\end{lemma}

\begin{proof}
  For the proof of \eqref{LinStr-radial}, see for example \cite[
  Theorem 1.3]{St05}. Then \eqref{LinStr-RadialTransfer} follows from
  \eqref{LinStr-radial} by applying the transfer principle in Lemma
  \ref{LemmaTransfer}.
\end{proof}

The following Lemma extends the result of Foschi-Klainerman for $m=0$ \cite[Lemma 4.4]{FK00} to the massive case.
\begin{lemma}\label{LemmaKG}
  Let $m\ge 0$ and consider the integral
 $$
 I(\phi, \psi)(\tau, \xi)= \int \phi(|\eta|)\psi(|\xi-\eta|)\delta(
 \tau-\angles{\eta}_m+ \angles{\xi-\eta}_m) \, d\eta.
 $$
 Then
 \begin{align}
   \label{I-intg}
   I(\phi, \psi)(\tau, \xi)&\simeq \frac1{|\xi|}
   \int_{\frac{\tau+|\xi|}2}^{\infty} \phi(\rho)\psi(\varphi
   (\tau,\rho)) \rho (\angles{\rho}_m-\tau) \, d\rho,
 \end{align}
 where
$$
\varphi (\tau, \rho)= \sqrt{ (\angles{\rho}_m-\tau)^2-m^2}.
$$
\end{lemma}
\begin{proof}
  The proof given here is a modification of the argument for $m=0$ from \cite[Lemma 4.4]{FK00}.
  For a smooth function $g$, define the hypersurface
$$S=\{x\in\R^3: g(x)=0\}. $$ If
$\nabla g \neq 0$ for $x\in S \cap \text{supp} f$, then
\begin{equation}\label{Delta1}
  \int f(x) \delta(g(x))\, dx=\int_S \frac{f(x)}{|\nabla g(x)|} \, dS_x.
\end{equation}
For a nonnegative smooth function $h$ which does not vanish on $S$,
\eqref{Delta1} also implies
\begin{equation}\label{Delta2}
  \delta(g(x))=h(x)\delta\left(h(x)g(x)\right).
\end{equation}

Now using \eqref{Delta2}, we can write
\begin{align*}
  \delta( \tau-\angles{\eta}_m+\angles{\xi-\eta}_m) &=[-(
  \tau-\angles{\eta}_m)+ \angles{\xi-\eta}_m] \delta\left(
    -(\tau-\angles{\eta}_m)^2+ \angles{\xi-\eta}_m^2\right)
  \\
  &=2( \angles{\eta}_m-\tau) \delta\left(|\xi|^2- \tau^2+2\tau
    \angles{\eta}_m - 2\xi \cdot \eta \right),
\end{align*}
where in the first line we multiplied the argument of the delta
function on the left by $ -(\tau-\angles{\eta}_m)+
\angles{\xi-\eta}_m$.
 
Introduce polar coordinate $\eta=\rho \omega$, where $\omega \in
\mathbb S^2$. Then $$ |\eta|=\rho, \quad d\eta= \rho^2 dS_w d\rho.$$
If we also set $a= \omega \cdot \xi/ |\xi|$, then $$dS_w = dS_{w'} da,
\quad d\eta= \rho^2 dS_{w'} d\rho da, $$ where $\omega' \in \mathbb
S^1$.  In the change of variables, we obtain $\angles{\xi-\eta}_m=
\angles{\eta}_m-\tau=\angles{\rho}_m-\tau$, which in turn implies
$|\xi-\eta|^2= (\angles{\rho}_m-\tau)^2-m^2.$ With these
transformations our integral becomes
$$
I(\phi, \psi)(\tau, \xi)\simeq \int_0^\infty \int_{-1}^{1}
\phi(\rho)\psi\left( \varphi(\tau, \rho) \right)
\rho^2(\angles{\rho}_m-\tau) \delta\left( |\xi|^2-\tau^2+2\tau
  \angles{\rho}_m- 2|\xi|\rho a \right)\, da d\rho.
 $$
 The delta function sets the value of $a$ to
 \begin{equation}\label{a1}
   a=\frac{ |\xi|^2-\tau^2+2\tau \angles{\rho}_m  }{ 2|\xi| \rho},
 \end{equation}
 which implies $a\ge \frac{\tau}{|\xi|}$, and thus we are restricted
 to $\frac{\tau}{|\xi|} \le a\le 1$.  This forces us to integrate over
$$
\left \{\rho >0 : \rho \ge \frac{ \tau+|\xi| }{ 2}\right\}.
$$
Using these facts and \eqref{Delta1} gives the desired estimate.
\end{proof}

\begin{lemma}\label{LemmaKG1}
  Let $m\ge 0$.  Consider $u^+(t)=S_m(t) f$ and $v^-(t)=S_m(-t) g$,
  where $f$ and $g$ are radial.  Then for any $\mu , \lambda_1,
  \lambda_2\ge 1 $, we have
  \begin{align*}
    \label{KGE}
    \norm{P_{ \mu} ( u^+_{\lambda_1} v^-_{\lambda_2}) }_{L^2_{t,
        x}(\R^{1+3})}&\lesssim \mu \norm{ f_{ \lambda_1}}_{L^2_{
        x}(\R^{3})} \norm{g_{ \lambda_2}}_{L^2_{ x}(\R^{3})} .
  \end{align*}

\end{lemma}

Then Lemma \ref{LemmaTransfer} and Lemma \ref{LemmaKG1} imply the
following Corollary.

\begin{corollary}\label{CorrKG2}
  Let $m\ge 0$. Suppose $u_{\lambda_1}$ and $v_{\lambda_2}$ are radial
  functions such that $u_{\lambda_1} , v_{\lambda_2} \in U^2_{m}$.
  Then for any $\mu , \lambda_1, \lambda_2\ge 1 $, we have
  \begin{equation*}
    \label{KGE-U2}
    \norm{ P_{ \mu} ( u_{\lambda_1} \overline{ v}_{\lambda_2}) }_{ L^2_{t, x}(\R^{1+3}) }\lesssim 
    \mu 
    \norm{ u_{ \lambda_1} }_{ U^2_{m} } \norm{ v_{\lambda_2} }_{ U^2_{m} }.
  \end{equation*}

\end{corollary}

 \begin{proof}[Proof of Lemma \ref{LemmaKG1}]
   First assume $\min(\lambda_1, \lambda_2)=1$. By symmetry we may
   assume $\lambda_1 \le \lambda_2$.  Then by H\"{o}lder,
   \eqref{LinStr-radial} and the energy inequality, we obtain
   \begin{align*}
     \norm{P_{ \mu} ( u^+_{\lambda_1} v^-_{\lambda_2}) }_{ L^2_{t, x}
     }&\lesssim \norm{ u^+_{\lambda_1} }_{ L^2_t L^\infty_x }\norm{
       v^-_{\lambda_2} }_{ L^\infty_t L^2_x }
     \\
     &\lesssim \norm{ f_{\lambda_1}}_{L^2_{ x}(\R^{3})} \norm{g_{
         \lambda_2}}_{L^2_{ x}(\R^{3})}.
   \end{align*}

   Therefore, we assume from now on that $\lambda_1 , \lambda_2 > 1$.
   If $f$ and $g$ are radial, then their Fourier transforms will also
   be radial.  Let $\widehat{f}(\xi)=\phi(|\xi|)$ and
   $\widehat{g}(\xi)=\psi(|\xi|)$.  Applying the space-time Fourier
   transform, we write
   \begin{align*}
     \widetilde{ P_\mu(u^+_{\lambda_1} v^-_{\lambda_2})}(\tau,
     \xi)&=\beta_{\mu}(|\xi|) \int_{\R^3} \beta_{\lambda_1}(|\eta|)
     \phi(|\eta|) \beta_{\lambda_2}(|\xi-\eta|)\psi(|\xi-\eta|)\delta(
     \tau-\angles{\eta}_m+\angles{\xi-\eta}_m) \, d\eta.
   \end{align*}
   Note that
 $$
 \tau^2=\left[\angles{\eta}_m-
   \angles{\xi-\eta}_m\right]^2=\left[\frac{|\eta|^2-|\xi-\eta|^2}
   {\angles{\eta}_m
     +\angles{\xi-\eta}_m}\right]^2=\left[\frac{|\xi|^2-2|\xi||\eta|
     \cos(\angle{(\xi, \eta))} } {\angles{\eta}_m
     +\angles{\xi-\eta}_m}\right]^2,
 $$
 which implies $$ |\tau|\leq\frac{|\xi|\left[|\xi|+2|\eta|\right]}
 {\angles{\eta}_m +\angles{\xi-\eta}_m} \lesssim \mu .
 $$
 Then using Lemma \ref{LemmaKG}, we obtain
 \begin{align*}
   & \norm{P_{ \mu} (u^+_{\lambda_1} v^-_{\lambda_2}) }_{L^2_{t,
       x}(\R^{1+3})}^2
   \\
   & \qquad \quad \lesssim \int_{\R^{3}} \int_{ |\tau|\lesssim \mu }
   \frac{\beta^2_{\mu}(|\xi|)}{|\xi|^2} \Bigabs{\int_{ \frac{
         \tau+|\xi| }{ 2} } ^{ \infty} [ \beta_{\lambda_1}(\rho) \rho
     \phi(\rho)][ \beta_{\lambda_2}(\varphi(\tau,\rho))
     (\angles{\rho}_m-\tau) \psi( \varphi(\tau,\rho))] \, d\rho}^2 \,
   d\tau d\xi
   \\
   & \qquad \quad \simeq \int_0^\infty \beta^2_{\mu}(r) \int_{
     |\tau|\lesssim \mu } \Bigabs{ \int_{ \frac{ \tau+ r }{ 2} } ^{
       \infty} [ \beta_{\lambda_1}(\rho) \rho \phi(\rho)][
     \beta_{\lambda_2}(\varphi(\tau,\rho)) (\angles{\rho}_m-\tau)
     \psi( \varphi(\tau,\rho))] \, d\rho }^2 \, d\tau dr
   \\
   & \qquad \quad \lesssim \mu \int_{ |\tau|\lesssim \mu}
   \left(\int_\R | \beta_{\lambda_1}(\rho) \rho \phi(\rho)|^2 \,
     d\rho\right) \left( \int_\R | \beta_{\lambda_2}(\varphi(\tau,
     \rho)) (\angles{\rho}_m-\tau) \psi\left(\varphi(\tau,
       \rho)\right)|^2 \, d\rho\right) \, d\tau,
 \end{align*}
 where to get the third line we used the change of variable $\xi=r
 \omega$, for $\omega \in S^2$, and to obtain the fourth inequality we
 used Cauchy-Schwarz with respect to $\rho$ and the fact that
 $\int_0^\infty \beta^2_{\mu}(r) \,dr \lesssim \mu$.
 
 We now use the change of variable $\rho \mapsto \sigma= \varphi(\tau,
 \rho)$, which implies
$$ \angles{\rho}_m-\tau =\angles{\sigma}_m, \quad  
d\rho= \frac{\sigma\angles{\rho}_m}{ \rho \angles{\sigma}_m }
d\sigma,$$ (where $\rho\sim \lambda_1$, since $\lambda_1>1$) to write
\begin{align*}
  \int_\R | \beta_{\lambda_2}(\varphi(\tau, \rho))
  (\angles{\rho}_m-\tau) \psi\left( \varphi(\tau, \rho)\right)|^2 \,
  d\rho &\simeq \frac{\angles{\lambda_1}_m} { \lambda_1 } \int_\R |
  \beta_{\lambda_2}(\sigma) \angles{\sigma}_m \psi\left( \sigma
  \right)|^2 \frac{ \sigma } { \angles{\sigma}_m } \, d\sigma
  \\
  & = \frac{\angles{\lambda_1}_m} { \lambda_1 } \int_\R |
  \beta_{\lambda_2}(\sigma) \sigma \psi\left( \sigma \right)|^2 \frac{
    \angles{\sigma}_m }{\sigma } \, d\sigma
  \\
  & \simeq \frac{\angles{\lambda_1}_m\angles{\lambda_2}_m} { \lambda_1
    \lambda_2} \int_\R | \beta_{\lambda_2}(\sigma) \sigma \psi\left(
    \sigma \right)|^2 \, d\sigma
  \\
  & \lesssim \int_\R | \beta_{\lambda_2}(\sigma) \sigma \psi\left(
    \sigma \right)|^2 \, d\sigma.
\end{align*}
 
We thus obtain
\begin{align*}
  \norm{ P_{ \mu} (u^+_{\lambda_1} v^-_{\lambda_2}) }_{ L^2_{t,
      x}(\R^{1+3}) }^2 &\lesssim \mu^2 \norm{\beta_{\lambda_1}(\rho)
    \rho \phi(\rho)}^2_{L_\rho^2 (\R)} \norm{\beta_{\lambda_2}(\rho)
    \rho \psi(\rho)}^2_{L_\rho^2 (\R)}
  \\
  & = \mu^2 \norm{ f_{\lambda_1}}^2_{L_x^2(\R^{3})} \norm{g_{
      \lambda_2}}^2_{L_x^2(\R^{3})},
\end{align*}
where to get the last equality we used the identities
  $$
  \norm{\beta_{\lambda_1}(\rho) \rho \phi(\rho)}_{L_\rho^2(\R)}=\norm{
    f_{\lambda_1}}_{L_x^2(\R^{3})}, \quad
  \norm{\beta_{\lambda_2}(\rho) \rho
    \psi(\rho)}_{L_\rho^2(\R)}=\norm{g_{ \lambda_2}}_{L_x^2(\R^{3})}.
 $$

\end{proof}

\subsection{Estimates in the non-radial case}
The first part of the next Lemma is well-known.
\begin{lemma}[KG-Strichartz]\label{Lemma-KGStr}
  Let $m> 0$ and $u(t)=S_m(t) f$.  Then
  \begin{equation}
    \label{KGStr}
    \norm{u_\lambda }_{L^2_t L^6_x (\R^{1+3})} \lesssim \lambda^\frac56 \norm{ f_{ \lambda}}_{L^2_x(\R^{3})}.
  \end{equation}

  Moreover, for all $u_\lambda \in U^2_{m} $, we have
  \begin{align}
    \label{KGStr-Transfer}
    \norm{ u_{\lambda} }_{ L^2_{t} L^6_{ x} (\R^{1+3}) }&\lesssim
    \lambda^\frac56 \norm{ u_{\lambda} }_{ U^2_{m} }.
  \end{align}
\end{lemma}

\begin{proof}
  For the proof of \eqref{KGStr}, see for example \cite{DF08}. The
  estimate \eqref{KGStr-Transfer} follows from \eqref{KGStr} by
  applying the transfer principle in Lemma \ref{LemmaTransfer}.
\end{proof}

\begin{lemma}[KG-Localized Strichartz]\label{KGr-Str}
  Let $m> 0$.  Consider $u_{ \lambda, \mu}(t)=S_m(t) f_{ \lambda,
    \mu}$, where $\widehat f_{ \lambda, \mu}$ is supported in a cube
  of side length $\mu$ at a distance $\sim \lambda$ from the origin,
  with $1 \lesssim \mu\lesssim \lambda$.  Then for all $\varepsilon
  >0$, we have
  \begin{align}
    \label{KGStr1}
    \norm{ u_{\lambda, \mu} }_{ L^2_{t} L^\infty_{ x} (\R^{1+3}) }
    \lesssim (\mu\lambda)^\frac12 \lambda^\varepsilon \norm{ f_{
        \lambda, \mu} }_{ L^2_{ x}(\R^{3} )}.
  \end{align}
  Moreover, for all $u_{\lambda, \mu} \in U^2_{m} $, we have
  \begin{align}
    \label{KGStr1-Transfer}
    \norm{ u_{\lambda, \mu} }_{ L^2_{t} L^\infty_{ x} (\R^{1+3}) }
    &\lesssim (\mu\lambda)^\frac12 \lambda^\varepsilon \norm{
      u_{\lambda, \mu} }_{ U^2_{m} }.
  \end{align}
\end{lemma}
\begin{proof}
  The estimate \eqref{KGStr1-Transfer} follows from \eqref{KGStr1} by
  applying the transfer principle in Lemma \ref{LemmaTransfer}. So we
  only prove \eqref{KGStr1}.

  Without loss of generality, we take $m=1$.  By a standard $TT^*$
  argument, \eqref{KGStr1} is equivalent to the estimate
  \begin{equation}
    \label{TTstar}
    \norm{  K_{\lambda, \mu} \star F }_{ L^2_{t} L^\infty_{ x} (\R^{1+3}) } \lesssim 
    \mu\lambda^{1+2\varepsilon}
    \norm{ F  }_{ L^2_{ t} L_x^1(\R^{1+3} )},
  \end{equation}
  where
  \begin{equation}
    \label{Km}
    K_{\lambda, \mu} (t, x)=\int e^{i(x\cdot \xi -t\angles{\xi })} |\beta_{\lambda, \mu}(\xi)|^2\, d\xi
  \end{equation}
  and $\star$ defines the space-time convolution.  The kernel
  $K_{\lambda, \mu}$ satisfies the following estimates (see
  \cite[Lemma 2.2. Eq. (2.8)]{BH13} and \cite[Eq. (3.2)]{BH14}):
  \begin{align*}
    |K_{\lambda, \mu} (t, x)|&\lesssim \lambda^4 \left( 1+ \lambda |t|
    \right)^{-\frac32},
    \\
    |K_{\lambda, \mu} (t, x)|&\lesssim \mu^3 \left( 1+
      \frac{\mu^2}{\lambda} |t| \right)^{-1}.
  \end{align*}
  We then interpolate these estimates, for $\theta\in (0,1)$, to
  obtain
  \begin{align*}
    |K_{\lambda, \mu} (t, x)|&\lesssim \lambda^{4\theta} \left( 1+
      \lambda |t| \right)^{-\frac32 \theta} \mu^{3-3\theta} \left( 1+
      \frac{\mu^2}{\lambda} |t| \right)^{-(1-\theta)}
    \\
    &\lesssim \mu^{3-3\theta} \lambda^{4\theta} \left( 1+
      \frac{\mu^2}\lambda |t| \right)^{\frac{-\theta-2} 2} ,
  \end{align*}
  where we used $\mu\lesssim \lambda$. Hence
  \begin{align*}
    \norm{K_{\lambda, \mu}}_{L_t^1L^\infty_x} &\lesssim
    \mu^{3-3\theta} \lambda^{4\theta} \int_{0}^\infty \left( 1+
      \frac{\mu^2}\lambda t \right)^{\frac{-\theta-2} 2} \, dt
    \\
    &\lesssim \mu^{3-3\theta} \lambda^{4\theta} \cdot \frac
    \lambda{\mu^2}\lesssim \mu \lambda^{1+4\theta},
  \end{align*}
  which implies (applying also Young's inequality in $t$ and $x$)
$$
\norm{ K_{\lambda, \mu} \star F }_{ L^2_{t} L^\infty_{ x} } \lesssim
\norm{ K_{\lambda, \mu} }_{ L^1_{t} L^\infty_{ x} } \norm{ F }_{
  L^2_{t} L^1_{ x} } \lesssim \mu\lambda^{1+4\theta} \norm{ F }_{
  L^2_{t} L^1_{ x} }.
$$
Choosing $\theta=\frac \varepsilon 2$ yields the desired estimate
\eqref{TTstar}.
\end{proof}

\begin{remark}\label{rmk-cubedecomp}
  Let $C_z=\mu z+ [0, \mu)^3 $, $z\in \Z^3$, be the collection of
  cubes (as in \cite{HL14}) which induce a disjoint covering of
  $\R^3$. Then we have
$$
\sum_{z\in \Z^3} \norm{P_{C_z} u_\lambda}^2_{ L^2_t
  L^\infty_x}\lesssim (\mu\lambda) \lambda^{2\varepsilon} \norm{
  u_{\lambda} }^2_{ U^2_{m} }.
$$
\end{remark}

\section{Proof of the Main Theorem}\label{ProofMainThm}

Similarly to \cite{HHK09, KT07}, we define $X^s $ to be the complete
space of all functions $u:\R\to L^2$ such that $P_\mu u \in
U^2_{m}(\R, L^2) $ for all $\mu\ge 1$, with the norm
$$  \norm{u}_{X^s } =
\left( \sum_{\mu\ge 1} \mu^{2s} \norm{ P_\mu u }_{ U^2_{m } }^2
\right)^{ \frac{1}{2} } < \infty,$$ \text{where } $$\norm{f}_{ U^2_{m
  } } = \norm{S_m(-t) f}_{ U^2}. $$ We also define by $Y^s$ the
corresponding space where $U^2_m$ is replaced by $V^2_m$ with a norm
 $$\norm{f}_{ V^2_{m } } = \norm{S_m(-t) f}_{ V^2}. $$

 On the time interval $I=[0, \infty)$, we define the restricted space
 $X^s_I$ by
$$ X^s_I = \Big\{ u \in C(I, H^s)\ |\; \tilde{u} = \chi_{I}(t)u(t) \in X^s \Big\} $$
with norm $$ \norm{u}_{X^s_I} = \norm{\chi_{I}u}_{X^s}, $$ and define
$Y^s_I$ analogously. Note the embedding
$$
X^s \subset Y^s.
$$

The Duhamel representation of \eqref{BS} is given by
\begin{equation}\label{IntegralEq}
  u(t) = S_m(t) f- i J_m(u)(t)
\end{equation}
where \begin{equation}\label{Duhamel} J_m(u)(t) = \int_0^t S_m(t-t')
  [( V\ast |u|^2) u] (t') dt'.
\end{equation} 
 
The linear part satisfies the following estimate ($m\ge 0$):
\begin{align*}
  \norm{ S_m(t) f}^2_{X^s_I}&=\sum_{\mu\ge 1} \mu^{2s} \norm{ \chi_I
    S_m(t) P_\mu f }_{ U^2_{m } }^2
  \\
  &=\sum_{\mu\ge 1} \mu^{2s} \norm{ \chi_I P_\mu u }_{ U^2}^2 \sim
  \norm{f}^2_{H^s}.
\end{align*}

Theorem \ref{MainThm} will follow by contraction argument from the
above linear estimate and following nonlinear estimates.
 
\begin{proposition}\label{Prop-WpEst}
  For $J_m(u) $ as in \eqref{Duhamel}, we have the following:
  \begin{enumerate}[(a)]
  \item \label{keyprop:a} Let $m\ge 0$ and $s>0$. For all spatially
    radial $u\in X^s_I $,
    \begin{align*}
      \norm{J_m(u)}_{X^s_I} &\lesssim \norm{u}_{X^s_I}^3.
    \end{align*}
  \item \label{keyprop:b} Let $m>0$ and $s>\frac12$. For all $u\in
    X^s_I $,
    \begin{align*}
      \norm{J_m(u)}_{X^s_I} &\lesssim \norm{u}_{X^s_I}^3.
    \end{align*}
  \end{enumerate}
\end{proposition}

The proof for this proposition is given in the last section.

 \subsection*{Proof of Theorem \ref{MainThm}}
 We solve the integral equation \eqref{IntegralEq} by contraction
 mapping techniques as follows.  Define the mapping
 \begin{equation}\label{contrmap}
   u(t) =  \Phi (u) (t):= S_m(t)f+ iJ_m(u)(t).
 \end{equation}
 We look for the solution in the set \[ D_\delta = \{ u \in X^s_I:\
 \|u\|_{X^s_I} \leq \delta \}. \] For $u \in D_\delta$ and initial
 data of size $\norm{f}_{H^s}\le \varepsilon\ll \delta$, we have by
 Proposition \ref{Prop-WpEst},
$$
\|\Phi(u)\|_{ X^s_I} \lesssim \varepsilon + \delta^3 \leq \delta
$$
for small enough $\delta$. Moreover, for solutions $u$ and $v$ with
the same data, one can show the difference estimate
\[\begin{aligned} \|\Phi(u)-\Phi(v)\|_{X^s_I} &\lesssim (\|u\|_{X^s_I}+\|v\|_{X^s})^2\|u-v\|_{X^s_I} \\
  &\lesssim \delta^2 \|u-v\|_{X^s_I} \end{aligned}\] whenever $u, \ v
\in D_\delta$.  Hence $\Phi$ is a contraction on $D_\delta$ when
$\delta \ll 1$, which implies the existence of a unique fixed point in
$D_\delta$ solving the integral equation \eqref{contrmap}.

It thus remains to show scattering of solution of \eqref{contrmap} to
a free solution as $t\rightarrow \infty$. By Proposition \ref{Prop-Vp}
and Proposition \ref{Prop-WpEst}, we have for each $\mu$
$$
S_m(-t) P_\mu J_m(u) \in V^2_{-, \text{rc}}
$$
and hence the limit as $t\rightarrow \infty$ exists for each
$\mu$. Combining this with
$$
\sum_{\mu\ge 1} \mu^{2s} \norm{ P_\mu J_m(u) }_{ V^2}^2 \lesssim 1
$$
implies that
$$
\lim_{t\rightarrow \infty} S_m(-t) P_\mu J_m(u)=: \phi^+ \in H^s.
$$
Hence for the solution $u$ we have that
$$
\norm {S_m(t)\phi^+-u(t)}_{ H^s} \rightarrow 0 \ \ \text{as} \
t\rightarrow \infty.
$$

\section{Proof of Proposition \ref{Prop-WpEst} } \label{ProofKeyProp}
We may assume that $u(t)=0$ for $-\infty < t<0 $.
 
By duality (see e.g. \cite{HHK09}),
\begin{equation*}\begin{aligned}\label{duality}
    \norm{P_\lambda J_m(u)}_{U^2_{m}} &= \norm{S_m(-t) P_\lambda
      J_m(u)}_{U^2} =
    \norm{P_\lambda \int_0^t S_m(-t')  [( V\ast |u|^2) u] (t')dt'}_{U^2} \\
    &= \sup_{ \norm {v}_{V^2}=1 } \Big|\iint  [( V\ast |u|^2) u] (t)\ \overline{  S_m(t) P_\lambda v(t)}\ d t dx\Big| \\
    &= \sup_{\norm {v_\lambda}_{V^2_{m}}=1} \Big|\iint [( V\ast |u|^2)
    u] (t) \overline{ v}_\lambda(t)\ dt dx\Big|,
  \end{aligned}
\end{equation*}
Then
\begin{equation}
  \label{Jm-X}
  \begin{split}
    \norm{J_m(u)}^2_{X^s} & =\sum_{\lambda\ge 1} \lambda^{2s} \norm{
      P_{\lambda} J_m(u) }_{ U^2_{m } }^2
    \\
    &=\sum_{ \lambda\ge 1 } \lambda^{2s} \sup_{ \norm{ v_{\lambda}}_{
        V^2_{m }}=1} \Big|\iint [ V\ast(u \overline{u} ) ] u \
    \overline{ v}_{\lambda}\ dt dx\Big| ^2
    \\
    &\lesssim \sum_{ \lambda\ge 1 } \lambda^{2s} \sup_{ \norm{
        v_{\lambda}}_{ V^2_{m }}=1} \left[ \sum_{\lambda_1, \lambda_2,
        \lambda_3\ge 1} \Big|\iint [ V\ast (u_{\lambda_1}
      \overline{u}_{\lambda_2} ) ] u_{\lambda_3} \ \overline{
        v}_{\lambda}\ dt dx\Big|\right]^2
  \end{split}
\end{equation}

To finish the proof of Proposition \ref{Prop-WpEst} \eqref{keyprop:a},
we use the following Lemma which we shall prove in the last section.
\begin{lemma}\label{KeyLemma1}
  Let $\lambda_{\mathrm{min}}$ and $\lambda_{\mathrm{med}}$ denote the
  minimum and median of $(\lambda_1, \lambda_2, \lambda_3)$.
  \begin{enumerate}[(a)]
  \item\label{keylemma:a} Let $m\ge 0$, $\varepsilon>0$. For all
    spatially radial $ u_{j, \lambda_j}\in U^2_{m} $, $v_\lambda \in
    V^2_{m} $ and $\lambda_j, \lambda \ge 1$,
    \begin{align*}
      \Big|\iint [ V\ast (u_{1, \lambda_1} \overline{u}_{2, \lambda_2}
      ) ] u_{3, \lambda_3} \ \overline{ v}_{\lambda}\ dt dx\Big| &
      \lesssim (\lambda_{\mathrm{min}} \lambda_{\mathrm{med}})
      ^\varepsilon \prod_{j=1}^3\norm { u_{j, \lambda_j}}_{U^2_{m}}
      \norm { v_{ \lambda}}_{V^2_{m}},
    \end{align*}
  \item\label{keylemma:b} Let $m>0$ and assume $ u_{j, \lambda_j}\in
    U^2_{m} $ and $v\in V^2_{m} $. Then, for all $\lambda_j, \lambda
    \ge 1$ and $\varepsilon>0$,
    \begin{align*}
      \Big|\iint [ V\ast (u_{1, \lambda_1} \overline{u}_{2, \lambda_2}
      ) ] u_{3, \lambda_3} \ \overline{ v}_{\lambda}\ dt dx\Big| &
      \lesssim (\lambda_{\mathrm{min}} \lambda_{\mathrm{med}})
      ^{\frac12+\varepsilon} \prod_{j=1}^3\norm { u_{j,
          \lambda_j}}_{U^2_{m}} \norm { v_{ \lambda}}_{V^2_{m}}.
    \end{align*}
  \end{enumerate}

\end{lemma}
  
We now finish the proof of Proposition \ref{Prop-WpEst}\ref{keyprop:a}
by using Lemma \ref{KeyLemma1}\ref{keylemma:a}.  The proof of
Proposition \ref{Prop-WpEst}\ref{keyprop:b} follows from a similar
argument using Lemma \ref{KeyLemma1}\ref{keylemma:b}.

Applying Lemma \ref{KeyLemma1}\ref{keylemma:a} to \eqref{Jm-X}, we
obtain
\begin{align*}
  \norm{J_m(u)}^2_{X^s} &\lesssim \sum_{ \lambda\ge 1 } \lambda^{2s}
  \sup_{ \norm{ v_{\lambda}}_{ V^2_{m }}=1} \left[ \sum_{\lambda_1,
      \lambda_2, \lambda_3 \ge 1} (\lambda_{\text{min}}
    \lambda_{\text{med}}) ^\varepsilon \prod_{j=1}^3\norm { u_{
        \lambda_j}}_{U^2_{m}} \norm { v_{ \lambda}}_{V^2_{m}}\right]^2
  \\
  &:=S_1 + S_2 + S_3,
\end{align*}
where
   $$
   S_1= \sum_{ \substack{ \lambda, \lambda_1, \lambda_2,  \lambda_3\ge 1\\
       \lambda_3\sim \lambda}}, \quad S_2=\sum_{ \substack{ \lambda,
       \lambda_1, \lambda_2, \lambda_3\ge 1
       \\\lambda_3\ll \lambda}}, \quad S_3=\sum_{ \substack{ \lambda, \lambda_1, \lambda_2,  \lambda_3\ge 1\\
       \lambda_3\gg \lambda}}.
 $$

 First Consider $S_{1}$. We have
 \begin{align*}
   S_{1} &\lesssim \sum_{ \lambda\ge 1 } \lambda^{2s} \left[ \sum_{
       \substack{\lambda_1, \lambda_2, \lambda_3\ge 1
         \\
         \lambda_3\sim \lambda}} (\lambda_{1} \lambda_{2})
     ^\varepsilon \norm { u_{ \lambda_1}}_{U^2_{m}} \norm { u_{
         \lambda_2}}_{U^2_{m}} \norm { u_{ \lambda_3}}_{U^2_{m}}
   \right]^2
   \\
   &\lesssim \norm{u_1}^2_{X^s} \norm{u_2}^2_{X^s} \sum_{ \lambda\ge 1
   } \lambda^{2s} \left[ \sum_{\lambda_3\sim \lambda } \norm { u_{
         \lambda_3}}_{U^2_{m}} \right]^2
   \\
   &\lesssim \norm{u}^2_{X^s} \norm{u}^2_{X^s} \norm{u}^2_{X^s} ,
 \end{align*}
 where to obtain the second inequality, we used Cauchy-Schwarz in
 $\lambda_1$ and in $\lambda_2$, and the fact that $ \sum_{
   \lambda_j\ge 1} \lambda_j^{-2(s-\varepsilon)} \lesssim 1$, if we
 assume $s>\varepsilon$.

 Next we consider $S_{2}$.  We have $S_2\le S_{21} + S_{22} + S_{23},$
 where
 $$
 S_{21}=\sum_{ \substack{\lambda_1\ll \lambda_2 \\ \lambda_3\ll
     \lambda} }, \quad S_{22}=\sum_{ \substack{\lambda_1\gg \lambda_2
     \\ \lambda_3\ll \lambda}}, \quad S_{23}=\sum_{
   \substack{\lambda_1\sim \lambda_2 \\ \lambda_3\ll \lambda}}
 $$

 If $\lambda_1 \ll \lambda_2$, then $\lambda_2\sim \lambda$.  Then we
 can apply Cauchy-Schwarz in $\lambda_1$ and in $\lambda_3$ to obtain
 \begin{align*}
   S_{21} &\sim \sum_{ \lambda } \lambda^{2s} \left[ \sum_{
       \substack{\lambda_1\ll\lambda_2\\ \lambda_3\ll \lambda\sim
         \lambda_2 }} (\lambda_{1} \lambda_{3}) ^\varepsilon \norm {
       u_{ \lambda_1}}_{U^2_{m}} \norm { u_{ \lambda_2}}_{U^2_{m}}
     \norm { u_{ \lambda_3}}_{U^2_{m}} \right]^2
   \\
   &\lesssim \norm{u}^2_{X^s} \norm{u}^2_{X^s} \sum_{ \lambda }
   \lambda^{2s} \left[ \sum_{\lambda_2\sim \lambda} \norm { u_{
         \lambda_2}}_{U^2_{m}} \right]^2
   \\
   &\lesssim \norm{u}^2_{X^s} \norm{u}^2_{X^s} \norm{u}^2_{X^s} ,
 \end{align*}
 The estimate for $S_{22}$ is similar.

 If $ \lambda_1 \sim \lambda_2$,then $\lambda \lesssim \lambda_1 \sim
 \lambda_2$. We apply Cauchy-Schwarz in $\lambda_1\sim \lambda_2$ and
 in $\lambda_3$ to obtain
 \begin{align*}
   S_{23} &\lesssim \sum_{ \lambda \ge 1} \lambda^{2s} \left[ \sum_{
       \substack{ \lambda_1\sim \lambda_2\\ \lambda_3\ll
         \lambda\lesssim \lambda_2} } (\lambda_{1} \lambda_{3})
     ^\varepsilon \norm { u_{ \lambda_1}}_{U^2_{m}} \norm { u_{
         \lambda_2}}_{U^2_{m}} \norm { u_{ \lambda_3}}_{U^2_{m}}
   \right]^2
   \\
   &\lesssim \norm{u}^2_{X^s} \norm{u}^2_{X^s} \norm{u}^2_{X^s} \sum_{
     \lambda\ge 1} \lambda^{2s} \left[ \lambda^{\varepsilon-2s}
   \right]^2
   \\
   &\lesssim \norm{u}^2_{X^s} \norm{u}^2_{X^s} \norm{u}^2_{X^s} .
 \end{align*}

 Finally, we consider $S_{3}$. As in the previous case, we have
 $S_3\le S_{31} + S_{32} + S_{33},$ where
 $$
 S_{31}=\sum_{ \substack{\lambda_1\ll \lambda_2 \\ \lambda_3\gg
     \lambda}}, \quad S_{32}=\sum_{ \substack{\lambda_1\gg \lambda_2
     \\ \lambda_3\gg \lambda}}, \quad S_{33}=\sum_{
   \substack{\lambda_1\sim \lambda_2 \\ \lambda_3\gg \lambda}}
 $$

 If $\lambda_1 \ll \lambda_2$, then $\lambda_2\sim \lambda_3$. We
 apply Cauchy-Schwarz first in $\lambda_1$ and then in $\lambda_2\sim
 \lambda_3$ to obtain
 \begin{align*}
   S_{31} &\lesssim \sum_{ \lambda \ge 1} \lambda^{2s} \left[ \sum_{
       \substack{\lambda_1\ll \lambda_2\\ \lambda \ll \lambda_3\sim
         \lambda_2 }} (\lambda_{1} \lambda_{2}) ^\varepsilon \norm {
       u_{ \lambda_1}}_{U^2_{m}} \norm { u_{ \lambda_2}}_{U^2_{m}}
     \norm { u_{ \lambda_3}}_{U^2_{m}} \right]^2
   \\
   &\lesssim \norm{u}^2_{X^s} \norm{u}^2_{X^s} \norm{u}^2_{X^s} \sum_{
     \lambda } \lambda^{2s} \left[ \lambda^{\varepsilon-2s} \right]^2
   \\
   &\lesssim \norm{u}^2_{X^s} \norm{u}^2_{X^s} \norm{u}^2_{X^s} .
 \end{align*}
 The estimate for $S_{32}$ is similar.
 
 If $\lambda_1 \sim \lambda_2$, then $\lambda_3 \lesssim \lambda_2
 $. We apply Cauchy-Schwarz in $\lambda_1\sim \lambda_2$ and in
 $\lambda_3$ to obtain
 \begin{align*}
   S_{33} &\lesssim \sum_{ \lambda } \lambda^{2s}
   \left[ \sum_{ \substack{\lambda_1\sim \lambda_2\\
         \lambda_3\gg\lambda }} (\lambda_{1} \lambda_{3}) ^\varepsilon
     \norm { u_{ \lambda_1}}_{U^2_{m}} \norm { u_{
         \lambda_2}}_{U^2_{m}} \norm { u_{ \lambda_3}}_{U^2_{m}}
   \right]^2
   \\
   &\lesssim \norm{u}^2_{X^s} \norm{u}^2_{X^s} \norm{u}^2_{X^s} \sum_{
     \lambda\ge 1 } \lambda^{2s} \left[\lambda^{\varepsilon-2s}
   \right]^2
   \\
   &\lesssim \norm{u}^2_{X^s} \norm{u}^2_{X^s} \norm{u}^2_{X^s} .
 \end{align*}

 \section{Proof of Lemma \ref{KeyLemma1}}\label{ProofKeyLemma}
   
   \subsection{Proof of Lemma \ref{KeyLemma1}\ref{keyprop:a} }
   First, we claim the following estimates for spatially radial
   functions $ u_{j, \lambda_j}\in U^2_{m} $, $v_\lambda \in V^2_{m} $
   and $\lambda,\lambda_j>0$, $\varepsilon>0$ and $p>2$:
   \begin{align}\label{KES1}
     \Big|\iint [ V\ast (u_{1, \lambda_1} \overline{u}_{2, \lambda_2}
     ) ] u_{3, \lambda_3} \ \overline{ v}_{\lambda}\ dt dx\Big| &
     \lesssim \min(\lambda_1, \lambda_2)^\varepsilon
     \prod_{j=1}^3\norm { u_{j, \lambda_j}}_{U^2_{m}} \norm{
       v_{\lambda}}_{U^2_{m}},
     \\
     \label{KES2}
     \Big|\iint [ V\ast(u_{1, \lambda_1} \overline{u}_{2, \lambda_2} )
     ] u_{3, \lambda_3} \ \overline{ v}_{\lambda}\ dt dx\Big| &
     \lesssim \lambda_{\text{min}} \lambda_{\text{med}}
     \prod_{j=1}^3\norm { u_{j, \lambda_j}}_{U^2_{m}} \norm { v_{
         \lambda}}_{U^p_{m}},
   \end{align}
   where $\varepsilon>0$ and $p>2$.
     
   Let us for the time being assume that the claims \eqref{KES1} and
   \eqref{KES2} hold. Given $M\ge 1$, we use Lemma \ref{Lemma:Interp}
   to decompose $ v_{ \lambda}$ into $ v_{ \lambda}= w_{ \lambda}+
   z_{\lambda}$, where $w_{ \lambda} \in U^2_{m}$ and $z_{ \lambda}
   \in U^p_{m}$, such that
   \begin{align}\label{InterpoEst}
     \frac{\kappa}M \norm { w_{\lambda}}_{U^2_{m}}+ e^M \norm { z_{
         \lambda}}_{U^p_{m}}\lesssim \norm { v_{\lambda}}_{V^2_{m}}
   \end{align}
   for some $\kappa>0$. We can assume that $\widetilde P_\lambda
   w_\lambda=w_\lambda$ and $\widetilde P_\lambda z_\lambda
   =z_\lambda$.

   We then use \eqref{KES1}--\eqref{InterpoEst} to obtain
   \begin{align*}
     \Big|\iint [ V\ast (u_{1, \lambda_1} \overline{u}_{2, \lambda_2}
     ) ] u_{3, \lambda_3} \ \overline{ v}_{\lambda}\ dt dx\Big| & \le
     \Big|\iint [ V\ast (u_{1, \lambda_1} \overline{u}_{1, \lambda_2}
     ) ] u_{1, \lambda_3} \ \overline{ w}_{\lambda}\ dt dx\Big|
     \\
     &\qquad \qquad \qquad + \Big|\iint [ V\ast (u_{1, \lambda_1}
     \overline{u}_{1, \lambda_2} ) ] u_{1, \lambda_3} \ \overline{
       z}_{\lambda}\ dt dx\Big|
     \\
     &\lesssim \min(\lambda_1, \lambda_2)^\varepsilon \prod_{j=1}^3
     \norm { u_{j, \lambda_j}}_{U^2_{m}} \norm {
       w_{\lambda}}_{U^2_{m}}
     \\
     &\qquad \qquad \qquad + \lambda_{\text{min}} \lambda_{\text{med}}
     \prod_{j=1}^3\norm { u_{j, \lambda_j}}_{U^2_{m}} \norm {
       z_{\lambda}}_{U^p_{m}}
     \\
     &\lesssim C_{\lambda, \kappa, M} \prod_{j=1}^3\norm { u_{j,
         \lambda_j}}_{U^2_{m}} \norm { u_{ \lambda}}_{V^2_{m}},
   \end{align*}
   where
   $$
   C_{\lambda, M}=\left\{\frac M\kappa \min(\lambda_1,
     \lambda_2)^\varepsilon + e^{-M} \lambda_{\text{min}}
     \lambda_{\text{med}} \right\}.
   $$
   Now if we choose
 $$
 M = \ln\left( \lambda_{\text{min}} \lambda_{\text{med}} \right),
 $$
 then
   $$
   C_{\lambda, \kappa, M} =\frac 1\kappa \ln\left(\lambda_{\text{min}}
     \lambda_{\text{med}} \right) \min(\lambda_1,
   \lambda_2)^\varepsilon +1 \lesssim (\lambda_{\text{min}}
   \lambda_{\text{med}} )^\varepsilon,
   $$
   and therefore the desired estimate follows.
  
   So it remains to prove \eqref{KES1} and \eqref{KES2}, which we do
   in the following subsections.

 \subsubsection{Proof of \eqref{KES1}}
 In $\R^3$ convolution with $V$ is (up to a multiplicative constant)
 the Fourier-multiplier $\angles{D}^{-2}$ with symbol
 $\angles{\xi}^{-2}$.  Using Littlewood-Paley decomposition and
 Cauchy-Schwarz, we obtain
 \begin{align*}
   \text{LHS of} \ \eqref{KES1} & \lesssim \sum_{\mu\ge 1
   }\angles{\mu}^{-2} \norm{ P_\mu(u_{1, \lambda_1} \overline{u}_{2,
       \lambda_2}) }_{L^2} \norm{ P_\mu(u_{3, \lambda_3}\ \overline{
       v}_{ \lambda}) }_{L^2}
   \\
   &:=S_1 + S_2 + S_3,
 \end{align*}
 where
$$
S_1= \sum_{1\le \mu \ll \lambda_1 }, \quad S_2=\sum_{ 1\le \mu \sim
  \lambda_1 }, \quad S_3= \sum_{ \mu \gg \lambda_1 }, \quad
$$
We now use Corollary \ref{CorrKG2} to obtain
\begin{align*}
  S_1&\lesssim \sum_{ 1\le \mu \ll \lambda_1 }
  \frac{\mu^2}{\angles{\mu}^{2} } \prod_{j=1}^3 \norm{ u_{ j,
      \lambda_j} }_{ U^2_{m} } \norm{ v_{\lambda} }_{ U^2_{m} }
  \\
  &\lesssim \lambda_1^\varepsilon \prod_{j=1}^3\norm { u_{j,
      \lambda_j}}_{U^2_{m}} \norm{ v_{\lambda} }_{ U^2_{m} },
  \\
  S_2 &\lesssim \sum_{1\le \mu \sim \lambda_1 }
  \frac{\mu^2}{\angles{\mu}^{2} } \prod_{j=1}^3 \norm{ u_{ j,
      \lambda_j} }_{ U^2_{m} } \norm{ v_{\lambda} }_{ U^2_{m} }
  \\
  &\lesssim \prod_{j=1}^3\norm { u_{j, \lambda_j}}_{U^2_{m}}\norm{
    v_{\lambda} }_{ U^2_{m} },
\end{align*}
where in the case of $\mu\ll \lambda_1$ we used
      $$ \sum_{ 1\le \mu \ll \lambda_1 } 
      \frac{ \mu }{ \angles{\mu}} \lesssim \ln(\lambda_1) +1\lesssim
      \lambda_1^\varepsilon $$ whereas if $ \mu\sim \lambda_1$ the sum
      is finite, i.e., $ \sum_{ 1\le \mu \sim \lambda_1 } \frac{ \mu
      }{ \angles{\mu}} \lesssim 1$.  Similarly,
      \begin{align*}
        S_3&\lesssim \prod_{j=1}^3\norm { u_{j,
            \lambda_j}}_{U^2_{m}}\norm{ v_{\lambda} }_{ U^2_{m} },
      \end{align*}

      \subsubsection{Proof of \eqref{KES2}}
  
      By symmetry (we do not exploit complex conjugation), we assume
      $\lambda_1\le \lambda_2\le \lambda_3$.  Since $V$ is bounded in
      $L^1(\R^3)$, we have by H\"{o}lder inequality,
      \eqref{LinStr-RadialTransfer} and Proposition \ref{Prop-Up}, we
      have for all $p\ge 2$,
      \begin{align*}
        \text{LHS of} \ \eqref{KES2} &\lesssim \norm { u_{1,
            \lambda_1}}_{L_t^2L_x^\infty}\norm { u_{2,
            \lambda_2}}_{L_t^2L_x^\infty} \norm { u_{3,
            \lambda_3}}_{L_t^\infty L_x^2}\norm {
          v_{\lambda}}_{L_t^\infty L_x^2} .
        \\
        &\lesssim \lambda_1 \lambda_2\norm { u_{1,
            \lambda_1}}_{U^2_{m} }\norm { u_{2, \lambda_2}}_{ U^2_{m}}
        \norm { u_{3, \lambda_3}}_{U^2_{m}}\norm { v_{
            \lambda}}_{U^p_{m}}.
      \end{align*}
  
      \subsection{Proof of Lemma \ref{KeyLemma1}\ref{keylemma:b}}

      Let $m>0$ and assume $ u_{j, \lambda_j}\in U^2_{m} $ and $v\in
      V^2_{m} $.

      Using Littlewood-Paley decomposition, we get
      \begin{align*}
        \Big|\iint [ V\ast (u_{1, \lambda_1} \overline{u}_{2,
          \lambda_2} ) ] u_{3, \lambda_3} \ \overline{ v}_{\lambda}\
        dt dx\Big| \lesssim K:= \sum_{\mu\ge 1 } \angles{\mu}^{-2}
        K_\mu ,
      \end{align*}
      where
 $$
 K_\mu=\norm{ P_\mu(u_{1, \lambda_1} \overline{u}_{2, \lambda_2} )
   P_\mu(u_{3, \lambda_3} \overline{ v}_{\lambda})}_{L^1_{t,x}}.
 $$
 
 By symmetry, we may assume $\lambda_1\le \lambda_2$.
 
 \subsubsection{Case 1: $\lambda_1\le \lambda_2\le \lambda_3$}
 \begin{enumerate}[(a)]
 \item \label{it:case11} Assume $\lambda_1\ll \lambda_2$. Then
   $\mu\sim \lambda_2$.  By Sobolev, H\"{o}lder,
   \eqref{KGStr1-Transfer} and Lemma \ref{LemmaTransfer},
   \begin{align*}
     K_\mu &\lesssim \norm { P_\mu (u_{1, \lambda_1} \overline{u}_{2,
         \lambda_2} ) }_{L_t^1L_x^\infty}
     \norm { u_{3, \lambda_3}}_{L_t^\infty L_x^2}\norm { v_{\lambda}}_{L_t^\infty L_x^2} \\
     & \lesssim \mu \norm { u_{1, \lambda_1} }_{L_t^2L_x^6} \norm {
       u_{2, \lambda_2} }_{L_t^2L_x^6} \norm { u_{3,
         \lambda_3}}_{L_t^\infty L_x^2}\norm {
       v_{\lambda}}_{L_t^\infty L_x^2}
     \\
     &\lesssim \mu (\lambda_1\lambda_2)^\frac 56 \prod_{j=1}^3\norm {
       u_{j, \lambda_j}}_{U^2_{m}} \norm { v_{ \lambda} }_{ V^2_{m} }
   \end{align*}
   Then (using $\mu\sim \lambda_2$)
   \begin{align*}
     K &\lesssim \sum_{\mu\ge 1 } \angles{\mu}^{-2}
     \mu^{2-2\varepsilon}(\lambda_1\lambda_2)^{\frac 13 +\varepsilon}
     \prod_{j=1}^3\norm { u_{j, \lambda_j}}_{U^2_{m}} \norm { v_{
         \lambda} }_{ V^2_{m} }
     \\
     &\lesssim (\lambda_1\lambda_2)^{\frac 13 +\varepsilon}
     \prod_{j=1}^3\norm { u_{j, \lambda_j}}_{U^2_{m}} \norm { v_{
         \lambda} }_{ V^2_{m} }.
   \end{align*}
  
 \item \label{it:case12} Assume $\lambda_1\sim \lambda_2\gg \mu$. In
   this case, we consider (as in \cite{HL14}) the collection of cubes
   $C_z=\mu z+ [0, \mu)^3 $, $z\in \Z^3$, which induce a disjoint
   covering of $\R^3$.  By the triangle inequality, we have
   \begin{align*}
     K_\mu &\lesssim \sum_{z, z'\in \Z^3} \norm{ P_\mu ( P_{C_z} u_{1,
         \lambda_1} P_{C_{z'}} \overline{u}_{2, \lambda_2} )
       P_\mu(u_{3, \lambda_3} \overline{ v}_{\lambda})}_{L^1_{t,x}}
   \end{align*}
   The point of this decomposition is that the term $ P_\mu ( P_{C_z}
   u_{1, \lambda_1} P_{C_{z'}} \overline{u}_{2, \lambda_2} )$ is zero
   for most of $z, z' \in \Z^3$, i.e., for each $z\in \Z^3$, only
   those $z'\in \Z^3$ with $|z-z'| \lesssim 1$ yield a nontrivial
   contribution on the sum.  We use this fact, H\"{o}lder,
   \eqref{KGStr1-Transfer} and Remark \ref{rmk-cubedecomp} to obtain
   \begin{align*}
     K_\mu &\lesssim \sum_{z, z'\in \Z^3} \norm { P_{C_z} u_{1,
         \lambda_1} }_{ L_t^2L_x^\infty} \norm { P_{C_{z'}} u_{2,
         \lambda_2} }_{L_t^2L_x^\infty} \norm { u_{3,
         \lambda_3}}_{L_t^\infty L_x^2}\norm {
       v_{\lambda}}_{L_t^\infty L_x^2}
     \\
     &\lesssim \left(\sum_{z\in \Z^3} \norm { P_{C_z} u_{1, \lambda_1}
       }^2_{ L_t^2L_x^\infty } \right)^\frac12 \left(\sum_{z' \in
         \Z^3} \norm { P_{C_{z'}} u_{2, \lambda_2} }^2_{
         L_t^2L_x^\infty }\right)^\frac12 \norm { u_{3, \lambda_3} }_{
       U^2_{m} } \norm { v_{ \lambda} }_{ V^2_{m} }
     \\
     &\lesssim \mu (\lambda_1\lambda_2)^{\frac 12 +\varepsilon}
     \prod_{j=1}^3\norm { u_{j, \lambda_j}}_{U^2_{m}} \norm { v_{
         \lambda} }_{ V^2_{m} }.
   \end{align*}
   Hence
   \begin{align*}
     K &\lesssim \sum_{\mu\ge 1 } \angles{\mu}^{-2}
     \mu(\lambda_1\lambda_2)^{\frac 12 +\varepsilon}
     \prod_{j=1}^3\norm { u_{j, \lambda_j}}_{U^2_{m}} \norm { v_{
         \lambda} }_{ U^4_{m} }
     \\
     &\lesssim (\lambda_1\lambda_2)^{\frac 12 +\varepsilon}
     \prod_{j=1}^3\norm { u_{j, \lambda_j}}_{U^2_{m}} \norm { v_{
         \lambda} }_{ V^2_{m} }
     \\
     &\lesssim (\lambda_1\lambda_2)^{\frac 12 +\varepsilon}
     \prod_{j=1}^3\norm { u_{j, \lambda_j}}_{U^2_{m}} \norm { v_{
         \lambda} }_{ V^2_{m} }.
   \end{align*}

 \end{enumerate}

 \subsubsection{Case 2: $\lambda_1\le \lambda_3 \le \lambda_2$ or $\lambda_3 \le \lambda_1\le  \lambda_2$} 
 First assume $\lambda_1\le \lambda_3 \le \lambda_2$.  If $\lambda_3
 \ll \lambda_2$ then we follow the same argument as in Case 1
 \eqref{it:case11}. If $\lambda_3 \sim \lambda_2$, then by following
 the same argument in Case 1 \eqref{it:case12}, we can obtain a
 maximum of $(\lambda_1\lambda_3)^{\frac 12 +\varepsilon}$ factor in
 the estimate for $K$.

 Next assume $\lambda_3 \le \lambda_1\le \lambda_2$. The only issue in
 this case is when $\lambda_3 \ll \lambda_1\sim \lambda_2$. If
 $\lambda_3\lesssim \mu$, we can then use the same argument as in Case
 1 \eqref{it:case11} by putting both $u_{1, \lambda_1}$ and $u_{3,
   \lambda_3}$ in $L_t^2L_x^6$ and the others in $L_t^\infty L_x^2$.
 It remains to consider the case $\lambda_3\gg \mu$, which also
 implies $\lambda \sim \lambda_3 \ll \lambda_1\sim \lambda_2$.  Let
 $\varepsilon'> 0$.  Arguing as in Case 1 \eqref{it:case12} (i.e.,
 using the fact $\lambda_1\sim \lambda_2\gg \mu$ ), we have
 \begin{align}\label{K1}
   K &\lesssim (\lambda_1 \lambda_2)^{\frac 12 +\varepsilon'}
   \prod_{j=1}^3\norm { u_{j, \lambda_j}}_{U^2_{m}} \norm { v_{
       \lambda} }_{ V^2_{m} }.
 \end{align}

 Similarly, we apply the decomposition in $\lambda_3 \sim \lambda$ as
 in Case 1 \eqref{it:case12} by putting both $u_{3, \lambda_3}$ and
 $v_{\lambda}$ in $L_t^2L_x^\infty$, whereas $u_{1, \lambda_1}$ and
 $u_{2, \lambda_2}$ in $L_t^\infty L_x^2$. Doing so we obtain
 \begin{align}\label{K2}
   K &\lesssim (\lambda \lambda_3)^{\frac 12 +\varepsilon'}
   \prod_{j=1}^3\norm { u_{j, \lambda_j}}_{U^2_{m}} \norm { v_{
       \lambda} }_{ U^2_{m} }.
 \end{align}

 Now we argue as in the proof of Lemma \ref{KeyLemma1}\ref{keylemma:a}
 by decomposing $ v_{ \lambda}= w_{ \lambda}+ z_{\lambda}$, where $w_{
   \lambda} \in U^2_{m}$ and $z_{ \lambda} \in U^p_{m}$, such that
 \eqref{InterpoEst} is satisfied. Then using \eqref{K1}--\eqref{K2},
 we obtain
 \begin{align*}
   K &\lesssim C_{\kappa, \lambda, M} \prod_{j=1}^3\norm { u_{j,
       \lambda_j}}_{U^2_{m}} \norm { u_{ \lambda}}_{V^2_{m}},
 \end{align*}
 where
   $$
   C_{\lambda, M}=\left[ \frac M\kappa (\lambda \lambda_3)^{\frac12
       +\varepsilon' }+ e^{-M} (\lambda_{1}
     \lambda_{2})^{\frac12+\varepsilon'} \right].
   $$
   Now if we choose
 $$
 M = \left(\frac12+\varepsilon'\right) \ln\left( \lambda_1\lambda_2
 \right),
 $$
 and $\varepsilon'=\varepsilon/2$, then
 \begin{align*}
   C_{\lambda, \kappa, M} &=\frac{1+\varepsilon}{2\kappa} [\ln\left(
     \lambda_1\lambda_2 \right) ]
   (\lambda\lambda_{3})^{\frac{1+\varepsilon}2} +1
   \\
   &\lesssim \lambda_{1}^{\varepsilon} \lambda_{3}^{1+\varepsilon}
 \end{align*}
 which is in fact better than the desired estimate.

 \bibliographystyle{plain} \bibliography{global-hartree}

\end{document}